\documentclass[12pt]{amsart}
\usepackage[utf8]{inputenc}

\usepackage{tikz-cd}
\usepackage{verbatim} %
\usepackage{graphicx}
\usepackage{mathtools}
\usepackage[all]{xy} 
\usepackage{mathrsfs} 
\usepackage{enumerate}
\usepackage[T2A]{fontenc}
\usepackage[utf8]{inputenc}
\usepackage{MnSymbol}

\usepackage{cleveref}

\DeclareMathAlphabet\mathbb{U}{msb}{m}{n}

\usepackage[OT2,T1]{fontenc} 
\DeclareSymbolFont{cyrletters}{OT2}{wncyr}{m}{n}

\DeclareMathOperator{\Sym}{{\sf Sym}}

\DeclareMathOperator{\Tor}{\textsl{Tor}}
\DeclareMathOperator{\Ker}{{\sf Ker}}
\DeclareMathOperator{\im}{{\rm I}{\sf m}}

\def\id{\textsl{id}}
\def\F2{{\mathbb{F}_2}}
\def\ZZ{\mathbb{Z}}

\def\ttt{\textsl{t}}

\def\epi{\twoheadrightarrow}
\def\mono{\rightarrowtail}
\def\AA{{\mathscr A}}

\numberwithin{equation}{section}

\newtheorem{theorem}{Theorem}[section]
\newtheorem{corollary}[theorem]{Corollary}
\newtheorem{lemma}[theorem]{Lemma}
\newtheorem{proposition}[theorem]{Proposition}
\theoremstyle{definition}

\begin{document}

\thanks{This work is supported by the Russian Science Foundation under grant 16-11-10073}

\title{Mod-$2$ (co)homology of an abelian group}
\author{Sergei O. Ivanov}

\address{
Laboratory of Modern Algebra and Applications,  St. Petersburg State University, 14th Line, 29b,
Saint Petersburg, 199178 Russia}
\email{ivanov.s.o.1986@gmail.com}

\author{Anatolii Zaikovskii}
\address{Laboratory of Modern Algebra and Applications, St. Petersburg State University, 14th Line, 29b,
Saint Petersburg, 199178 Russia and St. Petersburg Department of
Steklov Mathematical Institute} \email{anat097@mail.ru}

\dedicatory{Dedicated to Alexander Generalov on the occasion of his 70th birthday.}

\begin{abstract} It is known that for a prime $p\ne 2$ there is the following natural description of the homology algebra of an abelian group $H_*(A,\mathbb F_p)\cong \Lambda(A/p)\otimes \Gamma({}_pA)$ and for finitely generated abelian groups there is the following description of the cohomology algebra of $H^*(A,\mathbb F_p)\cong \Lambda((A/p)^\vee)\otimes {\sf Sym}(({}_pA)^\vee).$
We prove that there are no such  descriptions for $p=2$ that ``depend'' only on $A/2$ and ${}_2A$ but we provide natural descriptions of $H_*(A,\mathbb F_2)$ and $H^*(A,\mathbb F_2)$ that ``depend'' on  $A/2,$ ${}_2A$ and a linear map $\tilde \beta:{}_2A\to A/2.$
Moreover, we prove that there is a filtration by subfunctors on $H_n(A,\mathbb F_2)$ whose quotients are $\Lambda^{n-2i}(A/2)\otimes \Gamma^i({}_2A)$ and that for finitely generated abelian groups there is a natural filtration on $H^n(A,\mathbb F_2)$ whose quotients are $ \Lambda^{n-2i}((A/2)^\vee)\otimes {\sf Sym}^i(({}_2A)^\vee).$ 
\end{abstract}

\maketitle

\section*{\bf Introduction}

Let $K$ be a commutative ring and $A$ be an abelian group.  Then the summation homomorphism $A\oplus A\to A, (a,b)\mapsto a+b$ induces a map $H_*(A,K)\otimes H_*(A,K)\to H_*(A,K)$ called Pontryagin product, which gives a structure of a graded supercommutative algebra on $H_*(A,K)$ (see \cite[Ch. V]{br}). Moreover, the diagonal map $A\to A\oplus A$ and the sign map $A\to A$ induce a comultiplication and an antipode on $H_*(A,K)$ which gives a structure of a graded Hopf algebra on $H_*(A,K).$ For a prime $p$ we set
\[
A/p:=A/pA\cong A\otimes \ZZ/p, \hspace{1cm} {}_pA:=\{a\in A\mid pa=0\}\cong \Tor(A,\ZZ/p).
\]
For a vector space $V$ we denote by $V[n]$ the graded vector space concentrated in the degree $n.$
In work of H. Cartan \cite{Cartan1} it was proved (and exposed later in details in  \cite{Cartan}, see also \cite[Ch. V. Th. 6.6]{br}) that for a prime $p\ne 2$ there is a natural isomorphism of graded algebras
\begin{equation}\label{eq_iso_for_odd}
H_*(A,\mathbb F_p)\cong \Lambda (A/p[1])\otimes \Gamma({}_pA[2]),
\end{equation}    
where $\Lambda$ denotes the exterior algebra over $\mathbb F_p$, $\Gamma$ denotes the divided power algebra over $\mathbb F_p.$  In fact,  this is a natural isomorphism of Hopf algebras. Since for finitely generated abelian groups cohomology is dual to homology, we obtain that there is an isomorphism 
$$H^*(A,\mathbb F_p)\cong \Lambda((A/p)^\vee[1])\otimes \Sym(({}_pA)^\vee [2]),$$
where $(-)^\vee$ denotes the dual vector space. 
In particular, there is a natural isomorphism for each homology group 
\begin{equation}\label{eq_iso_for_n}
H_n(A,\mathbb F_p)\cong \bigoplus_{i=0}^{\lfloor n/2\rfloor } \Lambda^{n-2i} (A/p)\otimes \Gamma^i({}_p A),
\end{equation}
and in the case of finitely generated groups for each cohomology group:
$$
H^n(A,\mathbb F_p)\cong \bigoplus_{i=0}^{\lfloor n/2\rfloor } \Lambda^{n-2i} ((A/p)^\vee)\otimes \Sym^i(({}_p A)^\vee).
$$
The isomorphism \eqref{eq_iso_for_odd} means that the homology  algebra $H_*(A,\mathbb F_p)$ can be naturally recovered from the two vector spaces $A/p$ and ${}_pA$ as a graded Hopf algebra. Non-formally, the couple of vector spaces $(A/p,{}_pA)$ is the ``minimal information'' about $A$ that is  needed to describe the homology algebra naturally in the case $p\ne 2$. More formally, 
if we denote by $\textsl{Vect}^2$ the category of couples of $\mathbb F_p$-vector spaces and denote by $\ttt:\textsl{Ab}\to \textsl{Vect}^2$ the functor $A\mapsto (A/p,{}_pA),$ 
then there is a factorization of the functor $H_*(-,\mathbb F_p)$ to the category of graded Hopf algebras via the functor $\ttt.$
In particular, there is a factorization for each homology group.
\begin{equation}\label{eq_factor_n}
\begin{tikzcd}
\textsl{Ab} \arrow[rr,"{H_n(-,\mathbb F_p)}"] \arrow[dr,"\ttt"'] && \textsl{Vect} \\
 &\textsl{Vect}^2\arrow[ru]
\end{tikzcd}
\end{equation}
The isomorphism \eqref{eq_iso_for_n} gives a simple formula for $H_2(A,\mathbb F_p)$ when $p\ne 2:$
\[H_2(A,\mathbb F_p)\cong \Lambda^2(A/p)\oplus {}_pA.\]
However, for $p=2$ we only have a short exact sequence
\[
0 \longrightarrow \Lambda^2(A/2) \longrightarrow H_2(A,\F2) \longrightarrow {}_2A \longrightarrow 0,
\]
which does not split naturally \cite[\S 3]{Mikhailov}. So, the behavior of $H_*(A,\mathbb F_p)$ is more complicated, when $p=2.$

The goal of this paper is to work out the situation for $p=2$ and, in particular, give a natural description of the graded Hopf algebra $H_*(A,\F2)$ for arbitrary abelian group $A$ and the algebra  $H^*(A,\F2)$ for a finitely generated group $A.$ Moreover, our non-formal goal is to give a description that uses some ``minimal information'' 
about $A$ to recover naturally the homology algebra. We also obtain a natural filtration on
 $H_n(A,\F2)$ whose quotients are $\Lambda^i(A/2) \otimes \Gamma^j({}_2A),$ 
 and dually in the case of a finitely generated group $A$ a natural filtration on 
 $H^n(A,\F2)$ whose quotients are $\Lambda^i((A/2)^\vee) \otimes \Sym^j(({}_2A)^\vee).$ 
  This is an analog of the decomposition   \eqref{eq_iso_for_n} in the case of $p=2.$

We prove that there is no such factorization like \eqref{eq_factor_n} for any $n\geq 2$ in the case $p=2$ (\Cref{prop_non_existence}). So the couple of vector spaces $(A/2,{}_2A)$ is not enough information to recover naturally the Hopf algebras $H_*(A,\F2)$ and $H^*(A,\F2)$ or even the group $H_n(A,\F2)$ for some $n\geq 2.$ However, it is enough to add a linear map between these two vector spaces, to obtain the information which is enough to recover naturally the homology algebra. Namely, we need to add the map
\[
\tilde \beta:{}_2 A \longrightarrow A/2 
\]
which is the composition of the embedding ${}_2A \hookrightarrow A$ and the projection $A\twoheadrightarrow A/2.$

We describe the cohomology  of a finitely generated abelian group $A$ as follows. We set 
$$T(A)=(A/2)^\vee[1]\oplus ({}_2A)^\vee[2]$$ 
and prove that there is a natural isomorphism
$$H^*(A,\F2)\cong \Sym(T(A))/I,$$
where $I$ is the ideal generated by the set
$ \{x^2-\tilde \beta^\vee(x) \mid x\in (A/2)^\vee \}.$ We also prove that there is a unique structure of unstable $\mathscr A$-algebra  (where $\mathscr A$ is the Steenrod algebra) on $\Sym(T(A))/I$  such that $Sq^1(({}_2A)^\vee[2])=0$ and the isomorphism is an isomorphism of unstable $\mathscr A$-algebras. We also prove that there is a short exact sequence of bicommutative Hopf algebras
\[\F2 \longrightarrow \Sym(({}_2A)^\vee) \longrightarrow H^*(A,\F2) \longrightarrow \Lambda ((A/2)^\vee) \longrightarrow \F2. \]
  Moreover, we prove that the group $H^n(A,\F2)$ is naturally isomorphic to the cokernel of a map
\[\bigoplus_{2k+l+2m=n; \ k\geq 1} \Sym^k(A/2^\vee) \otimes \Sym^l(A/2^\vee)\otimes \Sym^m({}_2A^\vee) \longrightarrow \bigoplus_{i+2j=n} \Sym^i(A/2^\vee)\otimes \Sym^j({}_2A^\vee)  \]
and that
there is a natural filtration of $H^n(A,\F2)$ by subfunctors 
$\Phi^i$
such that 
$\Phi^i/\Phi^{i+1}\cong \Lambda^{n-2i}((A/2)^\vee)\otimes \Sym^i(({}_2A)^\vee).$
All the results can be generalized to the case of arbitrary abelian group $A$ if we consider cohomology $(A/2)^\vee$ and $ ({}_2A)^\vee$ as profinite vector spaces and replace all the constructions to their profinite versions.

We describe the homology of an abelian group $A$ (not necessary finitely generated) as follows. The following square is a pullback in the category of bicommutative Hopf algebras
$$
\begin{tikzcd} 
H_*(A,\F2)\arrow[rr] \arrow[d] && \Gamma({}_2A[2])\arrow[d,"\Gamma(\tilde \beta)"] \\
\Gamma(A/2[1])\arrow[rr,"{\mathcal V}"] && \Gamma(A/2[2]),
\end{tikzcd}
$$
where $\mathcal V$ denotes the Verschiebung. We also prove that there is a short exact sequence of graded bicommutative Hopf algebras
\[\F2 \longrightarrow \Lambda(A/2) \longrightarrow H_*(A,\F2) \longrightarrow \Gamma({}_2 A) \longrightarrow \F2.\]
Moreover 
$H_n(A,\F2)$ is naturally isomorphic to the kernel of a natural transformation 
 \[
\bigoplus_{i+2j=n} \Gamma^i(A/2)\otimes \Gamma^j({}_2A)
\longrightarrow 
\bigoplus_{2k+l+2m=n; \ k\geq 1} \Gamma^k(A/2) \otimes \Gamma^l(A/2)\otimes \Gamma^m({}_2A),  \]
 and
there is a natural filtration of $H_n(A,\F2)$ by subfunctors 
$\Psi_i$
such that 
$\Psi_i/\Psi_{i-1}\cong \Lambda^{n-2i}(A/2)\otimes \Gamma^i({}_2A).$

As an auxiliary result we prove that for a short exact sequence of graded connected bicommutative Hopf algebras
$$K \longrightarrow \mathcal A\longrightarrow \mathcal  B \longrightarrow \mathcal  C \longrightarrow K$$
over a field $K$
and  any $n\geq 0$ there is an isomorphism 
\[
\mathcal A_+^n\mathcal B/\mathcal A_+^{n+1}\mathcal B \cong \mathcal A_+^n/\mathcal A_+^{n+1} \otimes\mathcal  C,
\]
which depends only on the short exact sequence (without a linear splitting). Moreover, the isomorphism is natural by the short exact sequence. Note that the quotients depend only on $\mathcal A$ and $\mathcal C$ and do not depend on the short exact sequence. 

The first named author believes that statements proven in the paper are useful because his article \cite{Ivanov} could be twice shorter, if he could use them. We also think that there are some gaps in arguments in some articles that are filled by our results. For example, Bousfield in \cite{Bousfield92} uses the following statement without any reference: if $A$ is a module over a group $G$ such that $A/p$ and ${}_pA$ are nilpotent $G$-modules, then $H_n(A,\mathbb F_p)$ is a nilpotent $G$-module (see the proof of \cite[Prop. 2.8]{Bousfield92}). This statement is an obvious corollary of the result of Cartan for $p\ne 2.$ Apparently, Bousfild missed the case of $p=2,$ and our results fill the gap.

\section{\bf Filtration of a Hopf algebra associated with a short exact sequence}
In this section we prove that a short exact sequence of graded connected  bicommutative Hopf algebras $K\to\mathcal A \to\mathcal B \to\mathcal C \to K$ over a field $K$ provides a natural filtration on $\mathcal B$ whose quotients depend only on $\mathcal A$ and $\mathcal C$ and do not depend on the short exact sequence.

Let $\alpha : \mathcal A \to \tilde{\mathcal  A}$ be a morphism of augmented algebras.  Assume that $V$ is a vector space and $M$ is a $\tilde{\mathcal  A}$-module. Then for a linear map $f:V\to M$ we denote by $\bar  f$ the $\mathcal A$-module homomophism 
\[\bar  f :\mathcal  A\otimes V \longrightarrow M, \hspace{1cm} 
\bar f(a\otimes v)=\alpha(a)f(v).
\]
It is easy to see that 
$
\bar f(\mathcal A_+^n\otimes V)\subseteq \tilde{\mathcal  A}_+^nM.
$
Then the map $f$ induces a map on the quotients 
\[  f'_n :\mathcal A_+^n/\mathcal A_+^{n+1}\otimes V \longrightarrow \tilde{\mathcal A}_+^nM/\tilde{\mathcal A}_+^{n+1}M.\]
\begin{lemma}\label{lemma_barf=0}
If $\im (f)\subseteq \tilde{\mathcal A}_+M,$ then $f'_n=0$ for any $n\geq 0.$
\end{lemma}
\begin{proof}
Since $\im(f)\subseteq \tilde{\mathcal A}_+M,$ we have $\im(\bar f)\subseteq \tilde{\mathcal A}_+M.$ Using that $\bar f$ is a homomorphism of $A$-modules, we obtain that 
$\bar  f(\tilde{\mathcal A}_+^n\otimes V)\subseteq \alpha(\tilde{\mathcal A}_+^n)\cdot \im(\tilde f)\subseteq \tilde{\mathcal A}_+^{n+1}M.$ The assertion follows. 
\end{proof}

Assume now that we have a short exact sequence in the category bicommutative Hopf algebras
\begin{equation}\label{eq_sh_ex_seq}
K \longrightarrow \mathcal  A\overset{i}\longrightarrow \mathcal B \overset{\pi}\longrightarrow \mathcal C \longrightarrow K.
\end{equation}
and $\tilde {\mathcal A}={\mathcal A}, \alpha=\id_{\mathcal A}.$
Then ${\mathcal B}$ has a natural structure of an ${\mathcal A}$-module and for any linear map $f:C\to {\mathcal B}$ we obtain an ${\mathcal A}$-module homomorphism 
\[
\bar f : {\mathcal A}\otimes {\mathcal C} \longrightarrow \mathcal B. 
\]

\begin{lemma}\label{lemma_splitting}
Let \eqref{eq_sh_ex_seq}  be a short exact sequence of graded connected commutative and cocommutative Hopf algebras. Assume that $f:{\mathcal C}\to {\mathcal B}$ is a graded linear map such that $\pi f=\id_{\mathcal C}.$ Then 
\[
\bar  f:{\mathcal A}\otimes {\mathcal C} \xrightarrow{\cong } {\mathcal B}, \hspace{1cm} a\otimes c\mapsto a\cdot f(c) 
\]
is an isomorphism of ${\mathcal A}$-modules.  
\end{lemma}
\begin{proof} It follows from Proposition 1.7  of \cite{MM}. See also the proof of Theorem 4.4 and Proposition 4.9 of \cite{MM}. 
\end{proof}
\begin{theorem}\label{th_nat_filt_gen}
For a short exact sequence of graded connected bicommutative Hopf algebras
$$K \longrightarrow {\mathcal A}\xrightarrow{\ i \ } {\mathcal B} \xrightarrow{\ \pi \ } {\mathcal C} \longrightarrow K$$
and  any $n\geq 0$ there is an isomorphism 
\begin{equation}\label{eq_filt_iso_gen}
{\mathcal A}_+^n{\mathcal B}/{\mathcal A}_+^{n+1}{\mathcal B} \cong {\mathcal A}_+^n/{\mathcal A}_+^{n+1} \otimes {\mathcal C},
\end{equation}
which depends only on the short exact sequence (without a linear splitting). Moreover, the isomorphism is natural by the short exact sequence. 
\end{theorem}
\begin{proof}
If we take any graded linear splitting  $f:{\mathcal C}\to {\mathcal B}$ of the epimorphism $\pi:{\mathcal B}\to {\mathcal C},$ then by 
\Cref{lemma_splitting} we obtain that $\bar f:{\mathcal A}\otimes {\mathcal C} \cong {\mathcal B}$ is an isomorphism of ${\mathcal A}$-modules. It follows that $f'_n$ is an isomorphism of the form \eqref{eq_filt_iso_gen}. 

Prove that this isomorphism does not depend on the choice  of the splitting $f.$ The difference of two such splittings is a map $g:{\mathcal C}\to {\mathcal B}$ such that $\im g \subseteq \Ker \pi.$ Then it is enough to prove that for any linear map $g:{\mathcal C}\to {\mathcal B}$ such that $\im g \subseteq \Ker \pi={\mathcal A}_+{\mathcal B}$ we have $g'_n=0.$ This follows from \Cref{lemma_barf=0}
 So, the isomorphism \eqref{eq_filt_iso_gen} does not depend of the choice of $f.$ 
 
Prove that this isomorphism is natural by the short exact sequence. Assume that we have a morphism of short exact sequences
\[
\begin{tikzcd}
K\arrow[r] & {\mathcal A}\arrow[r]\arrow[d,"\alpha_{\mathcal A}"] & {\mathcal B}\arrow[r,"\pi"]\arrow[d,"\alpha_{\mathcal B}"] & {\mathcal C}\arrow[r]\arrow[d,"\alpha_{\mathcal C}"] & K\\
K \arrow[r]& \tilde {\mathcal A}\arrow[r] & \tilde {\mathcal B}\arrow[r,"\tilde \pi"] & \tilde {\mathcal C}\arrow[r] & K
\end{tikzcd}
\]
Fix a splitting $f$ of $\pi$ and a splitting $g$ of $\tilde \pi.$ Set $h=\alpha_{\mathcal B} f - g \alpha_{\mathcal C}.$
We claim that  $\im (h)\subseteq \Ker \tilde \pi.$ Indeed, ths follows from $\im(h)=\im(h\pi)$ and $\tilde \pi h\pi=\tilde \pi \alpha_{\mathcal B} f \pi - \tilde \pi g \alpha_{\mathcal C} \pi=\alpha_{\mathcal C} \pi h \pi - \tilde \pi g  \tilde \pi \alpha_{\mathcal B}=\alpha_{\mathcal C}\pi - \tilde \pi \alpha_{\mathcal B}=0.$ Then $\im(h)\subseteq \tilde {\mathcal A}_+\tilde {\mathcal B}.$ \Cref{lemma_barf=0} implies $h'_n=0.$ It follows that the diagram 
\[
\begin{tikzcd}
({\mathcal A}^n_+/{\mathcal A}^{n+1}_+)\otimes {\mathcal C}\arrow[r,"f'_n"]\arrow[d] & {\mathcal A}_+^n{\mathcal B} /{\mathcal A}_+^{n+1}{\mathcal B}  \arrow[d]\\
(\tilde {\mathcal A}_+^n/\tilde {\mathcal A}_+^{n+1})\otimes \tilde {\mathcal C}\arrow[r,"g'_n"] & \tilde {\mathcal A}_+^n\tilde {\mathcal B} /\tilde {\mathcal A}_+^{n+1}\tilde {\mathcal B}
\end{tikzcd}
\]
is commutative. The assertion follows. 
\end{proof}

\section{\bf Generalities about Bockstein homomorphism and the universal coefficient theorem}

For a space $X$ we denote by
$$\beta:H_n(X,\ZZ/p)\longrightarrow H_{n-1}(X,\ZZ/p)$$ 
the Bockstein homomorphism. The short exact sequence $\ZZ \overset{\cdot p}\mono \ZZ \epi \ZZ/p$ induces a boundary map
$\beta':H_n(X,\ZZ/p) \to H_{n-1}(X,\ZZ),$ whose composition with the map $H_{n-1}(X,\ZZ)\to H_{n-1}(X,\ZZ/p)$ is $\beta$ \cite[\S 3E]{Hatcher}.  The universal coefficient theorem says that there is a short exact sequence
\[ 0\longrightarrow H_n(X,\ZZ)/p \overset{\theta_n} \longrightarrow H_n(X,\ZZ/p) \overset{\tau_n}\longrightarrow 
{}_pH_{n-1}(X,\ZZ) \longrightarrow 0,\]
where the map $\tau_n$ is constructed as the restriction of the image of $\beta'$.
 \[
\begin{tikzcd}
H_{n}(X,\ZZ/p)\arrow[dr,dashed,"\tau_n"'] \arrow[rr,"\beta'"] && H_{n-1}(X,\ZZ) \arrow[r,"\cdot p"]  & H_{n-1}(X,\ZZ) \\
& {}_pH_{n-1}(X,\ZZ)\arrow[ru,hookrightarrow] & &  
\end{tikzcd}
\]
This follows that $\beta$ equals the following composition 
\[ 
\beta:\ \ \  H_n(X,\ZZ/p) \xrightarrow{\tau_n} {}_p H_{n-1}(X,\ZZ) \hookrightarrow H_{n-1}(X,\ZZ) \rightarrow H_{n-1}(X,\ZZ/p). 
\] 
The last homomorphism can be also decomposed as $H_{n-1}(X,\ZZ)\epi H_{n-1}(X,\ZZ)/p \xrightarrow{\theta_{n-1}} H_{n-1}(X,\ZZ/p).$ This implies the following lemma.

\begin{lemma}\label{lemma_bockstein} The Bockstein homomorphism $\beta$ decomposes via homomorphisms from the universal coefficient theorem as follows
\begin{equation}
\begin{tikzcd}
H_n(X,\ZZ/p)\arrow[d,"\tau_n",twoheadrightarrow] \arrow[r,"\beta"] & H_{n-1}(X,\ZZ/p) \\
{}_pH_{n-1}(X,\ZZ)\arrow[r,"\tilde \beta"] & H_{n-1}(X,\ZZ)/p \arrow[u,"\theta_{n-1}",rightarrowtail], 
\end{tikzcd}
\end{equation}
 where $\tilde \beta$ is the composition of the embedding ${}_pH_{n-1}(X,\ZZ) \hookrightarrow H_{n-1}(X,\ZZ)$ and the projection $H_{n-1}(X,\ZZ) \epi H_{n-1}(X,\ZZ)/p.$ 
\end{lemma}

\section{\bf Cohomology}

For an abelian group $A$ we set 
$T(A)=(A/2)^\vee[1] \oplus ({}_2A)^\vee[2].$
\begin{theorem}\label{th_cohom} Let $A$ be a finitely generated abelian group. Then
\begin{itemize} 
\item[(a)] There is a natural isomorphism of graded Hopf algebras
\begin{equation}\label{eq_th_morphism}
H^*(A,\F2)\cong \Sym(T(A))/I,
\end{equation}
where $I$ is the ideal generated by the set 
$\{ x^2 -  \tilde \beta^\vee(x) \mid x\in (A/2)^\vee \}.$
\item[(b)]
The following square is a pushout and pullback in the category of graded bicommutative Hopf algebras 
\begin{equation}\label{eq_th_square}
\begin{tikzcd}
\Sym((A/2)^\vee[2]) \arrow[d,"\Sym(\tilde \beta^\vee)"] \arrow[rr,"\mathcal F"] && \Sym((A/2)^\vee[1]) \arrow[d]\\
\Sym(({}_2A)^\vee[2]) \arrow[rr] && H^*(A,\F2),
\end{tikzcd}
\end{equation}
where $\mathcal F$ is the Frobenious homomorphism.

\item[(c)] There exists a unique structure of unstable $\mathscr A$-algebra on $\Sym(T(A))/I$ such that $Sq^1(x)=0$ for any $x\in ({}_2A)^\vee.$ The morphism \eqref{eq_th_morphism} is an isomorphism of unstable $\mathscr A$-algebras with respect to this structure. 
\end{itemize}
\end{theorem}
\begin{proof}[Proof of Theorem \ref{th_cohom}]
 Dualizing the universal coefficient theorem we obtain a short exact sequence
\[0 \longrightarrow ({}_2 H_{n-1}(A,\ZZ))^\vee \overset{\tau^n}\longrightarrow H^n(A,\F2) \overset{\theta^n}\longrightarrow (H_n(A,\ZZ)/2)^\vee \longrightarrow 0.
 \]
If we identify $H_1(A,\ZZ)= A,$ we get morphisms 
\begin{equation}\label{eq_maps_to_prim}
(\theta^1)^{-1}:(A/2)^\vee \overset{\cong}\longrightarrow H^1(A,\F2), \hspace{1cm} \tau^2:({}_2A)^\vee \mono H^2(A,\F2).
\end{equation}
These maps are natural by $A.$ The diagram of naturality of these morphisms with respect to the homomorphism of addition $A\oplus A \to A$ implies that the images of the maps \eqref{eq_maps_to_prim} consist of primitive elements. Then they give morphisms of Hopf algebras from symmetric algebras 
\begin{equation}\label{eq_homomorphisms_from_symmetric}
\Sym((A/2)^\vee[1]) \to H^*(A,\F2), \hspace{1cm} \Sym(({}_2A)^\vee[2]) \to H^*(A,\F2).
\end{equation}

The Bockstein homomorphism in degree one $\beta:H^1(A,\F2) \to H^2(A,\F2)$ equals the Frobenius homomorphism. Then Lemma \ref{lemma_bockstein} implies that the following diagram is commutative
\[
\begin{tikzcd}
(A/2)^\vee \arrow[d,"\cong"] \arrow[rr,"\tilde \beta^\vee"] && ({}_2A)^\vee \arrow[d,rightarrowtail]\\
H^1(A,\F2) \arrow[rr,"{x\mapsto x^2}"] && H^2(A,\F2).
\end{tikzcd}
\]
It follows that the square of Hopf algebras
\eqref{eq_th_square}
is commutative. It is known \cite[Th. 4.4]{Newman}, \cite[Cor. 4.16]{Takeuchi} that the category of bicommutative Hopf algebras is abelian, the direct sum is given by the tensor product and cokernel of a morphism $\alpha: C\to D$ is given by $D/\alpha(C_+)D,$ where $C_+$ is the augmentation ideal of $C.$ 
It is easy to check that a commutative square in the category of bicommutative graded Hopf algebras is pushout if and only if it is pushout in the category of bicommutative (non-graded) Hopf algebras. Therefore the pushout of the morphisms $\mathcal F$ and $\Sym(\tilde \beta)$ in the diagram \eqref{eq_th_square} is $\Sym(T(A))/I$ (here we use that in an abelian category the pushout of two maps is the obvious quotient of their direct sum and that $\Sym(T(A))\cong \Sym((A/2)^\vee[1])\otimes \Sym(({}_2A)^\vee[2])$). Therefore we obtain a natural morphism of graded Hopf algebras 
$$\Sym(T(A))/I\to H^*(A,\F2)$$
and the square \eqref{eq_th_square} is a pushout if and only if the morphism is an isomorphism.

All functors in the diagram  \eqref{eq_th_square} are additive as functors from the category of abelian groups to the category of bicommutative Hopf algebras (they send direct sums to tensor products). It follows that in order to prove that the square is a pushout for finitely generated abelian groups, it is enough to prove that the square is a pushout for cyclic groups $A=\ZZ,\ZZ/p^n,$ where $p$ is prime and $n\geq 1$. 
If $p\ne 2$ and $n\geq 1,$ then all algebras in the square are trivial. The case $A=\ZZ$ follows from the standard fact that $H^*(\ZZ,\F2)\cong \F2[x]/(x^2).$ The other cases $\ZZ/2$ and $\ZZ/2^n$ for $n\geq 2$  follow from the isomorphisms 
 $H^*(\ZZ/2,\F2)\cong \F2[x],$ $H^*(\ZZ/2^n,\F2)\cong \F2[x,y]/(x^2),$ where ${\sf deg}(x)=1$ and ${\sf deg}(y)=2 $ (see {\cite[\S 3.2]{Evens}}). 
 
Recall that a square in an abelian category is a pushout (pullback) if and only if its totalization is a right (left) exact sequence. So, if a square is a pushout, then it is a pullback if and only if the first arrow of its totalization is a monomorphism. Therefore the fact that our square is also a pullback follows from the fact that the Frobenius homomorphism $\Sym((A/2)^\vee)\to \Sym((A/2)^\vee)$ is a monomorphism. 
 
We proved (a) and (b). Now we prove (c). The Verschiebung on the Steenrod algebra $v:\mathscr A \to \mathscr A$ is an endomorphism that acts as follows $v(Sq^{2n})=Sq^n$ and $v(Sq^{2n+1})=0$ 
(see \cite[Ch. II., Prop. 3.5]{SteenrodEpstein}). For any vector space $V$ there is a unique structure of unstable $\mathscr A$-algebra on $\Sym(V[1])$ which comes from the isomorphism $\Sym(V[1])\cong H^*(V,\F2).$ The action via Verschiebung gives a structure of unstable $\mathscr A$-algebra on $\Sym(V[2]).$ Using
 the natural structure of Hopf algebra on $\mathscr A,$ we obtain that there is a natural structure of unstable $\mathscr A$-algebra on the tensor product $\Sym(V[1])\otimes \Sym(U[2])\cong \Sym(V[1]\oplus U[2])$ for any two vector spaces $V,U$. This gives a structure of $\AA$-module on $\Sym(T(A)).$ Prove that the ideal $I$ is closed under this action. It is enough to prove that the image of the generating set with respect to the action of $Sq^n$ lies in $I.$ This follows from the equations $Sq^1(x^2 - \tilde \beta^\vee (x) )=0$ and $Sq^2(x^2 - \tilde \beta^\vee (x) )=(x^2 - \tilde \beta^\vee (x))^2.$ So, we obtain a required structure of unstable $\AA$-algebra on $\Sym(T(A))/I.$ It is easy to see that it is unique.  In order to prove that the isomorphism $\Sym(T(A))/I\cong H^*(A,\F2)$ respects the action of $\AA,$ it is enough to prove if for elements from $T(A).$ Since they are both unstable $\AA$-algebras, the action of $Sq^1$ on $(A/2)^\vee[1]$ and the action of $Sq^2$ on $({}_2A)^\vee[2]$ are respected by the isomorphism. So we only need to prove that the isomorphism respects the action of $Sq^1$ on $({}_2A)^\vee[2]$, which is trivial. So we need to prove that $Sq^1({\rm I}{\sf m}\  (\tau^2_A:({}_2A)^\vee \to H^2(A,\F2) ))=0.$ Consider the vector space $V={}_2A.$ Then the embedding $V \hookrightarrow A$ induces a commutative square 
$$
\begin{tikzcd}
V^\vee\arrow[d,"\tau^2_V"]\arrow[rr,"{\sf id}"] && {}_2A^\vee\arrow[d,"\tau^2_A"] \\
H^2(V,\F2)\arrow[rr] && H^2(A,\F2)
\end{tikzcd}
$$
Hence the assertion follows from the equation $Sq^1(\tau^2_V(x))=Sq^1Sq^1(x')=0,$ where $x'$ is the image of $x$ under the isomorphism $V^\vee\cong H^1(V,\F2).$   
\end{proof}

\begin{corollary}\label{cor_short_exact} For a finitely generated abelian group $A$ there is a natural short exact sequence of Hopf algebras
\[\F2 \longrightarrow \Sym(({}_2A)^\vee) \longrightarrow H^*(A,\F2) \longrightarrow \Lambda ((A/2)^\vee) \longrightarrow \F2. \]
\end{corollary}
\begin{proof}
This follows from the fact that the sequence $\F2 \to \Sym(V) \overset{\mathcal F}\to \Sym(V) \to \Lambda(V) \to \F2$ is short exact in the category of bicommutative Hopf algebras for any vector space $V$ and the pushout from Theorem \ref{th_cohom}. 
\end{proof}

\begin{corollary} For a finitely generated abelian group $A$ there is a natural filtration of $H^n(A,\F2)$ by subfunctors 
$H^n(A,\F2)=\Phi^0 \supseteq \Phi^1 \supseteq \dots \supseteq \Phi^{\lfloor n/2 \rfloor+1 }=0 $
such that 
\[\Phi^i/\Phi^{i+1}\cong \Lambda^{n-2i}((A/2)^\vee)\otimes \Sym^i(({}_2A)^\vee).\]
\end{corollary}
\begin{proof}
This follows from Corollary \ref{cor_short_exact} and Theorem \ref{th_nat_filt_gen}.
\end{proof}

\begin{corollary}
For a finitely generated ablelian group $A$ the group $H^n(A,\F2)$ is naturally isomorphic to the cokernel of the map 
\[\bigoplus_{2k+l+2m=n; \ k\geq 1} \Sym^k(A/2^\vee) \otimes \Sym^l(A/2^\vee)\otimes \Sym^m({}_2A^\vee) \longrightarrow \bigoplus_{i+2j=n} \Sym^i(A/2^\vee)\otimes \Sym^j({}_2A^\vee).  \]
In particular, there is a short exact sequence
\[0 \longrightarrow (A/2)^\vee \longrightarrow \Sym^2(A/2^\vee)\oplus {}_2A^\vee \longrightarrow H^2(A,\F2) \longrightarrow 0, \]
where the first map is $x\mapsto (x^2,\tilde \beta^\vee(x)),$
and a short exact sequence 
$$ 0 \longrightarrow (A/2)^\vee \otimes (A/2)^\vee \longrightarrow \Sym^3((A/2)^\vee)\oplus (A/2)^\vee \otimes ({}_2A)^\vee \longrightarrow H^3(A,\F2) \longrightarrow 0,$$
where  $x\otimes y \mapsto (x^2y, y\otimes \tilde \beta^\vee(x)).$
\end{corollary}
\begin{proof}
Theorem \ref{th_cohom} implies that $H^*(A,\F2)$ is the cokernel of the morphism  $\Sym((A/2)^\vee[2]) \to \Sym((A/2)^\vee[1])\otimes \Sym(({}_2A)^\vee[2])$ in the category of bicommutative Hopf algebras. The assertion follows. 
\end{proof}

\section{\bf Homology}

In this section dual results for homology are presented. They hold in a more general setting: for arbitrary abelian group, not only for finitely generated. The reason why the results hold in a more general setting is that homology, divided powers,  tensor products and all other functors that occur in the statements commute with filtered colimits. The proof for all the statements is the following: first we prove the results for finitely generated abelian groups dualising the results for cohomology; and then use that all the functors commute with filtered colimits, and present an abelian group as a colimit of its finitely generated subgroups.  

For a vector space $V$ we denote by $\Gamma(V)=\bigoplus \Gamma^n(V)$ the divided power algebra of $V.$ It is a subalgebra of the shuffle algebra ${\sf Sh}(V)$ such that $\Gamma^n(V)$ is the space of invariants of $V^{\otimes n}$ under the action of the symmetric group. It is well known that, if $V$ is finite dimensional, then $\Gamma(V)=\Sym(V^\vee)^\vee.$  

\begin{theorem}\label{th_hom} Let $A$ be an abelian group (not necessary finitely generated). Then the following square is a pullback in the category of graded bicommutative Hopf algebras
$$
\begin{tikzcd} 
H_*(A,\F2)\arrow[rr] \arrow[d] && \Gamma({}_2A[2])\arrow[d,"\Gamma(\tilde \beta)"] \\
\Gamma(A/2[1])\arrow[rr,"{\mathcal V}"] && \Gamma(A/2[2]),
\end{tikzcd}
$$
where $\mathcal V$ denotes the Verschiebung.
\end{theorem}
\begin{corollary}\label{cor_hom_short_exact}
For an abelian group $A$ there is a short exact sequence of graded bicommutative Hopf algebras
\[\F2 \longrightarrow \Lambda(A/2) \longrightarrow H_*(A,\F2) \longrightarrow \Gamma({}_2 A) \longrightarrow \F2.\]
\end{corollary}

\begin{corollary}
For an abelian group $A$ there is a natural filtration of $H_n(A,\F2)$ by subfunctors 
$ 0=\Psi_{-1}\subseteq \Psi_0 \subseteq \dots \subseteq \Psi_{\lfloor n/2 \rfloor}=H_n(A,\F2)$
such that 
\[\Psi_i/\Psi_{i-1}\cong \Lambda^{n-2i}(A/2)\otimes \Gamma^i({}_2A).\]
\end{corollary}

\begin{corollary}
For an ablelian group $A$ the group $H_n(A,\F2)$ is naturally isomorphic to the kernel of a map 
\[
\bigoplus_{i+2j=n} \Gamma^i(A/2)\otimes \Gamma^j({}_2A)
\longrightarrow 
\bigoplus_{2k+l+2m=n; \ k\geq 1} \Gamma^k(A/2) \otimes \Gamma^l(A/2)\otimes \Gamma^m({}_2A) 
.  \]
In particular, there is a short exact sequence
\[0 
\longrightarrow 
H_2(A,\F2) 
\longrightarrow 
\Gamma^2(A/2)\oplus {}_2A 
\longrightarrow 
A/2 
\longrightarrow 0, \]
and a short exact sequence 
\[ 
0 
\longrightarrow 
H_3(A,\F2) 
\longrightarrow 
\Gamma^3(A/2)\oplus A/2 \otimes {}_2A 
\longrightarrow 
A/2 \otimes A/2 
\longrightarrow 
0.
\]
\end{corollary}

\section{\bf Non existence of a factorization via pairs of vector spaces}

\begin{proposition}\label{prop_non_existence} Let $\textsl{Vect}^2$ be the category of pairs of vector spaces over $\F2$ and $\ttt:\textsl{Ab}\to \textsl{Vect}^2$ be the functor $\ttt(A)=(A/2,{}_2A).$ Assume that $n\geq 2.$ Then there is no a functor $\Phi:\textsl{Vect}^2\to \textsl{Vect}$ such that $\Phi\circ \ttt \cong H_n(-,\F2).$
\[H_n(A,\F2)\not\cong \Phi(A/2,{}_2A). \]
Moreover, there is no such a functor even if we restrict the functors to the category of finitely generated abelian groups and the category of finite dimensional vector spaces. 
\end{proposition}
\begin{proof}
Here we use the theory of quadratic functors from the category of free finitely generated abelian groups ${\sf fAb}\to {\sf Ab} $ and quadratic modules; we use the equivalence of these two categories (see \cite{Baues}, \cite{HartlVespa}, \cite{Mikhailov}).

Prove the proposition for $n=2.$ Assume the contrary, that there is a functor $\Phi:\textsl{Vect}^2\to \textsl{Vect}$ such that there is a natural isomorphism $H_2(A,\F2)\cong \Phi(A/2,{}_2A).$ 
Take a free finitely generated abelian group $F.$ Then there is an isomorphism 
$
\Phi(F/2,0)\cong H_2(F,\F2) \cong \Lambda^2 (F/2), 
$
which is natural by $F.$ Note that the standard isomorphism $\Sym((F/2)^\vee)\cong H^*(F/2,\F2) $ implies that there is an natural isomorphism
$ H_*(F/2,\F2)\cong \Gamma (F/2).$
It follows that 
$
H_2(F/2,\F2)\cong \Gamma^2 (F/2).
$ Therefore $\Phi(F/2,F/2)\cong \Gamma^2(F/2).$ Since $(F/2,F/2)=(F/2,0)\oplus (0,F/2)$ in the category $\textsl{Vect}^2$, we obtain that $\Phi(F/2,0)\cong \Lambda^2 F/2$ is a natural retract $\Phi(F/2,F/2)\cong \Gamma^2(F/2).$ On the other hand it is easy to compute the quadratic $\ZZ$-module of $\Lambda^2 (F/2):$
$0 \to \ZZ/2 \to 0,$
the quadratic $\ZZ$-module of $\Gamma^2(F/2):$
$ \ZZ/2 \xrightarrow{ 1} \ZZ/2 \xrightarrow{0} \ZZ/2,$
and note that the first one is not a retract of the second one. This is a contradiction. So we proved the proposition for $n=2.$

Prove for $n>2$ by induction. Note that the K\"unneth theorem implies
$H_{n}(A\times \ZZ,\F2) =H_{n}(A,\F2)\oplus H_{n-1}(A,\F2).$
Then
$
H_{n-1}(A,\F2)=\textsl{Coker}(H_n(A,\F2)\to H_n(A\times \ZZ,\F2)).
$
It follows that the existence of such a functor $\Phi$ for $H_n$ implies the existence of such a functor for $H_{n-1}.$ The assertion follows. 
\end{proof}

\end{document}